\documentclass{amsart}
\usepackage[english]{babel}
\usepackage{amsfonts}
\usepackage{amsmath,amsthm,amssymb,latexsym}
\usepackage{color}
\usepackage{authblk}
\usepackage{graphicx}
\usepackage{tikz}
\usepackage{stmaryrd}
\usepackage{mathrsfs,upgreek}
\usepackage{eucal}
\usepackage{enumerate}
\usepackage{changes}
\usepackage{amsaddr}

\newtheorem{theorem}{Theorem}[section]
\newtheorem{lemma}[theorem]{Lemma}
\newtheorem{proposition}[theorem]{Proposition}
\newtheorem{corollary}[theorem]{Corollary}

\theoremstyle{definition}

\newtheorem{remark}[theorem]{Remark}

\numberwithin{equation}{section}
\usepackage[english]{babel}
\usepackage{amsfonts}
\usepackage{amsmath,amsthm,amssymb}
\usepackage{color}
\usepackage{graphicx}
\usepackage{tikz}
\usepackage{stmaryrd}
\usepackage{mathrsfs,upgreek}
\usepackage{eucal}
\usepackage{enumerate}
\usepackage{changes}

\newcommand{\op}[1]{\textrm{\upshape #1}}

\newcommand{\join}{\vee}

\newcommand{\meet}{\wedge}

\newcommand{\la}{\langle}
\newcommand{\ra}{\rangle}
\newcommand{\alg}[1]{{\textbf{\upshape #1}}}  %
\newcommand{\vv}[1]{\mathsf {#1}}

\renewcommand{\d}{\delta}
\newcommand{\f}{\varphi}

\newcommand{\e}{\varepsilon}
\renewcommand{\th}{\theta}
\renewcommand{\o}{\omega}

\newcommand{\sse}{\subseteq}

\newcommand{\app}{\approx}
\newcommand{\HH}{{\mathbf H}}  
\newcommand{\II}{{\mathbf I}} 
\newcommand{\SU}{{\mathbf S}} 
\newcommand{\PP}{{\mathbf P}}   
\newcommand{\VV}{{\mathbf V}}   
\newcommand{\QQ}{{\mathbf Q}}

\newcommand{\itemb}{\item[$\bullet$]}

\newcommand{\OO}{{\mathbf O}}
\newcommand{\MTL}{\vv M\vv T\vv L}

\newcommand{\BL}{\vv B\vv L}

\newcommand{\ib}{\item[$\bullet$]}

\newcommand{\Con}[1]{\operatorname{Con}(\alg #1)}

\newcommand{\vuc}[2]{#1_1,\dots,#1_{#2}}

\newcommand{\imp}{\rightarrow}

\newcommand{\cc}[1]{\mathcal{#1}}
\newcommand{\lextimes}{\times_l}
\newcommand{\WH}{\mathsf{WH}}
\newcommand{\ol}[1]{\overline {#1}}

\begin{document}
\title[Positive \L ukasiewicz logic]{Structurally complete finitary extensions of positive \L ukasiewicz logic}

\author{Paolo Aglian\`{o}}
\address{DIISM, Universit\`a di Siena, Siena, Italy}
\email{agliano@live.com}

\author{Francesco Manfucci}
\address{DIISM, Universit\`a di Siena, Siena, Italy}
\email{francesco.manfucci@student.unisi.it}

\begin{abstract} In this paper we study $\mathcal{MV}^+$, i.e. the positive fragment of {\L}ukasiewicz Multi-Valued Logic $\mathcal{MV}$.
In particular we describe all the finitary extensions of $\mathcal{MV}^+$ that are structurally complete and all the
axiomatic extensions of $\mathcal{MV}^+$ that are hereditarily structurally complete. Examples of hereditarily structurally complete finitary extensions and
non hereditarily structurally complete finitary extensions are provided.	
\end{abstract}
\maketitle

\section{Introduction}

A large class of substructural logics (i.e. logics that lack some structural rules) is given by the extensions of the {\em Full Lambek Calculus} $\mathcal{FL}$ (see \cite{GJKO} p. 76).  All extensions of $\mathcal{FL}$ have a primitive connective $\mathbf{0}$ that denotes {\em falsum}; so if $\mathcal{L}$ is an extension of $\mathcal{FL}$ it makes sense to study its {\em positive} fragment $\mathcal{L}^+$. The language of $\mathcal L^+$ is obtained by the lnguage of $\mathcal L$ by deleting $\mathbf{0}$ from the signature; the valid derivations in $\mathcal L^{+}$ are the valid derivations in $\mathcal L$, that contains only $0$-free formulas.

An important subfamily of substructural logics over $\mathcal{FL}$ consists of logics that satisfy both exchange and weakening but lack contraction; in this case the primitive connectives can be taken as $\join,\meet,\imp,\mathbf{0},\mathbf{1}$ where of course $\mathbf{1}$ represents {\em truth}. The minimal substructural logic of this kind is denoted by $\mathcal{FL}_{ew}$ and hence all the substructural logics lacking contraction are an extension of it.  Probably Intuitionistic Logic $\mathcal{IL}$ is the most famous of them but there are many others such as the monoidal logic $\mathcal{MTL}$ \cite{EstevaGodo2001}, Hajek's Basic Logic $\mathcal{BL}$ \cite{Hajek1998}, Dummet's Logic $\mathcal{GL}$ \cite{Hajek1998}, Product Logic $\mathcal{PL}$ \cite{Hajek1998} and, last but not least, \L ukasiewicz's Many Valued logic $\mathcal {MV}$ \cite{Komori1981}.

All extensions of $\mathcal{FL}$ are {\em algebraizable} with an {\em equivalent algebraic semantics} (in the sense of Blok-Pigozzi\cite{BlokPigozzi1989}) that is at least a quasivariety of algebras; by the same token all positive fragments are algebraizable as well and moreover the machinery of algebrization takes a very transparent form in both cases. The equivalent algebraic semantics of $\mathcal{FL}$ and $\mathcal{FL}^+$ are the variety of $\mathsf{FL}$-algebras \cite{GJKO} and the variety of residuated lattices \cite{JipsenTsinakis2002} respectively.
Algebraizability in Blok-Pigozzi's fashion entails that
the deducibility relation of an algebraizable logic is characterized by means of the algebraic
equational consequence of its equivalent algebraic semantics. Therefore, interesting properties
of algebraizable logics can be studied via algebraic means.

Hoops are a particular variety of residuated monoids related to logic which were defined in
an unpublished manuscript by B\"uchi and Owens, inspired by the work of on partially
ordered monoids in \cite{Bosbach1969}. Hoops indeed correspond to residuated commutative monoids whose order is the {\em inverse divisibility ordering}, i.e.
$a\le b$ if and only if there exists $c$ such
that $a = bc$. The first systematic study of hoops has been carried out by Ferreirim in her PhD
thesis \cite{Ferr1992} and in her works with Blok \cite{BlokFerr2000}; the connection between hoops and substructural logics has been investigated in \cite{AFM}.
A subvariety of hoops, the variety $\mathsf{WH}$ of {\em Wajsberg hoops}, plays a special role in this framework. From the algebraic point of view, they
can be used to describe subdirectly irreducible hoops, and the whole variety of hoops can be
obtained as the join of iterated powers of the variety of Wajsberg hoops, in the sense defined
in \cite{BlokFerr2000}. Wajsberg hoops also have a peculiar connection with lattice-ordered abelian groups
(abelian $\ell$-groups for short). In fact, the variety of Wajsberg hoops is generated by its
totally ordered members, that are, in loose terms, either negative cones of abelian $\ell$-groups, or
intervals in abelian $\ell$-groups \cite{AglianoPanti1999}. However the most relevant property of Wajsberg
hoops in this context is that they are the equivalent algebraic semantics of the positive fragment of \L ukasiewicz logic
$\mathcal{MV}$.  This will allow us to reduce logical questions, such as the structural completeness of its finitary extensions,
to a purely algebraic matter.

\section{Preliminaries}

\subsection{Universal algebra} \label{subsec:universal}  Let $\vv K$ be a class of algebras; we denote by $\II,\HH,\PP,\SU,\PP_u$ the class operators sending $\vv K$ in the class of all isomorphic copies, homomorphic images, direct products, subalgebras and ultraproducts of members of $\vv K$. The operators can and will be composed in the obvious way; for instance $\SU\PP(\vv K)$ denotes all algebras that are embeddable in a direct product of members of $\vv K$; moreover there are relations among the classes resulting from applying operators in a specific order, for instance $\PP\SU(\vv K) \sse \SU\PP(\vv K)$ and $\HH\SU\PP(\vv K)$ is the largest class we can obtain composing the operators. We will use all the known relations without further notice; for those and for any other unexplained algebraic notion we refer the reader to \cite{BurrisSanka} for a textbook treatment.

 If $\tau$ is a type of algebras, an {\bf equation} is a pair $p,q$ of $\tau$-terms (i.e. elements of the absolutely free algebra $\alg T_\tau(\o)$) that we write suggestively as $p \app q$; a {\bf universal sentence} in $\tau$ is   $\Sigma \Rightarrow \Gamma$ where $\Sigma,\Gamma$ are finite sets of equations; a universal sentence is a {\bf quasiequation} if $|\Gamma| = 1$ and it is {\bf negative} if $\Delta= \emptyset$. Clearly an equation is a quasiequation in which $\Sigma = \emptyset$.
 An equation $p(\vuc xn) \app q(\vuc xn)$ is {\bf valid} in $\alg A$ (and we write $\alg A \vDash \Sigma \Rightarrow \Delta$) if for all $\vuc an \in A$, $p(\vuc an) = q(\vuc an)$; if $\Sigma$ is a set of equations then
 $\alg A \vDash \Sigma$ if $\alg A \vDash \sigma$ for all $\sigma \in \Sigma$.
 A universal sentence is {\bf valid} in $\alg A$ (and we write $\alg A \vDash \Sigma \Rightarrow \Delta$) if whenever  $\alg A \vDash \Sigma$ there is a $\e \app \d \in \Delta$ with
 $\alg A \vDash \e \app \d$; in other words a universal sentence can be understood as the formula $\forall \mathbf x(\bigwedge \Sigma \imp \bigvee \Delta)$.  An equation or a universal sentence is {\bf valid} in a class $\vv K$ if it is valid in all algebras in $\vv K$.

A class of algebras is a variety if it is closed under $\HH, \SU$ and $\PP$,  a quasivariety if it is closed under $\II$,$\SU$,$\PP$ and $\PP_u$ and a universal class if it is closed under $\II\SU\PP_u$.
 The following facts were essentially discovered by A. Tarski, J. \L\`os and A. Lyndon in the pioneering phase of model theory; for proof of this and similar statements the reader can consult \cite{ChangKeisler}.

 \begin{lemma}\label{lemma:ISP}
 Let $\vv K$ be any class of algebras. Then:
 \begin{enumerate}
 \item  $\vv K$ is a universal class if and only if $\II\SU\PP_u(\vv K) = \vv K$ if and only if it is the class of all algebras in which a set  of universal sentences is valid;
\item  $\vv K$ is a quasivariety if and only if $\II\SU\PP\PP_u(\vv K) = \vv K$ if and only if it is the class of all  algebras in which a set  of quasiequations  is valid;
\item $\vv K$ is a variety if and only if $\HH\SU\PP(\vv K) = \vv K$ if and only if it is the class of all algebras in which a set of equations is valid.
\end{enumerate}
  \end{lemma}

We will often write $\VV$ for $\HH\SU\PP$ and $\QQ$ for $\II\SU\PP\PP_u$. It is clear that both $\VV$ and $\QQ$ are closure operators on classes of algebras of the same type; this implies among other things that for a given variety $\vv V$ the class of subvarieties of $\vv V$ is a complete lattice
which we denote by $\Lambda(\vv V)$. If $\vv Q$ is a quasivariety that is not a variety then $\Lambda(\vv Q)$ is still a lattice but it is not necessarily complete (in particular does not need to have maximum). The class of subquasivarieties of a quasivariety $\vv Q$  is a complete lattice denoted by
$\Lambda_q(\vv Q)$.

For the definition of free algebras in a class $\vv K$ on a set $X$ of generators in symbols $\alg F_\vv K(X)$, we refer again to \cite{BurrisSanka}. We merely observe that every free algebra on a class $\vv K$ belongs to
$\II\SU\PP(\vv K)$. It follows that every free algebra in $\vv K$ is free in $\II\SU\PP(\vv K)$ and therefore for any quasivariety $\vv Q$, $\alg F_\vv Q(X) = \alg F_{\VV(\vv Q)}(X)$.

Let $\alg B, (\alg A_i)_{i \in I}$ be algebras in the same signature; we say that $\alg B$ {\bf embeds} in $\prod_{i \in I} \alg A_i$  if $\alg B \in \II\SU(\prod_{i\in I} \alg A_i)$.
Let $p_i$ be the $i$-th projection, or better, the composition of the isomorphism and the $i$-th projection, from $\alg B$ to $\alg A_i$; the embedding is
{\bf subdirect} if for all $i \in I$, $p_i(\alg B) = \alg A_i$ and in this case we will write
$$
\alg B \le_{sd} \prod_{i \in I} \alg A_i.
$$
An algebra $\alg B$ is {\bf subdirectly irreducible} if it is nontrivial and for any subdirect embedding
$$
\alg B \le_{sd} \prod_{i \in I} \alg A_i.
$$
there is an $i \in I$ such that $\alg B$ and $\alg A_i$ are isomorphic.  It can be shown that $\alg A$ is {\bf subdirectly irreducible} if and only if the congruence lattice $\Con A$ of  $\alg A$ has a unique minimal element different from the trivial congruence. If $\vv V$ is a variety we denote by $\vv V_{si}$ the class of subdirectly irreducible algebras in $\vv V$.

\begin{theorem}\label{birkhoff} \begin{enumerate}
 \item (Birkhoff \cite{Birkhoff1944}) Every algebra can be subdirectly embedded in a product of subdirectly irreducible algebras. So if $\alg A \in \vv V$, then $\alg A$ can be subdirectly embedded in a product of members of $\vv V_{si}$.
\item (J\'onsson's Lemma \cite{Jonsson1967}) Suppose that $\vv K$ is a class of algebras such that $\vv V(\vv K)$ is congruence distributive;
then $\vv V_{si} \sse \HH\SU\PP_u(\vv K)$.
\end{enumerate}
\end{theorem}

If $\vv Q$ is a quasivariety and $\alg A \in \vv Q$, a {\bf $\vv Q$-congruence}  of $\alg A$ is a congruence $\th$ such that $\alg A/\th \in \vv Q$; clearly $\vv Q$-congruences
form an algebraic lattice $\op{Con}_\vv Q(\alg A)$. An algebra $\alg A\in \vv Q$ is {\bf $\vv Q$-irreducible} if $\op{Con}_\vv Q(\alg A)$ has a unique minimal element; since clearly $\op{Con}_\vv Q(\alg A)$ is a
meet subsemilattice of $\Con A$, any subdirectly irreducible algebra is $\vv Q$-irreducible in any quasivariety $\vv Q$ to which it belongs. For a quasivariety $\vv Q$ we denote by
$\vv Q_{ir}$ the class of $\vv Q$-irreducible algebras in $\vv Q$.
We have the equivalent of Birkhoff's and J\'onsson's results for quasivarieties:

\begin{theorem}\label{quasivariety} Let $\vv Q$ be any quasivariety.
\begin{enumerate}
\item (Mal'cev \cite{Malcev1956}) Every $\alg A \in \vv Q$ is subdirectly embeddable in a product of algebras  in $\vv Q_{ir}$.
\item (Czelakowski-Dziobiak \cite{CzelakowskiDziobiak1990}) If $\vv Q = \QQ(\vv K)$, then $\vv Q_{ir} \sse \II\SU\PP_u(\vv K)$.
\end{enumerate}
\end{theorem}

\subsection{The Blok-Pigozzi connection}

In what follows by a {\bf logic} $\mathcal{L}$  we mean a substitution invariant consequence relation $\vdash_\mathcal{L}$ on the set of terms $\alg T_{\tau}(\omega)$ (also called \emph{algebra of formulas}) of some algebraic language $\tau$.
A {\bf clause} in $\mathcal{L}$ is a formal expression $\Sigma \Rightarrow \Delta$ where $\Sigma,\Delta$ are finite set of formulas of $\mathcal L$; a clause is a {\bf rule} if $\Delta =\{\d\}$. A rule is an {\bf axiom} if $\Sigma = \emptyset$.

If the equivalence between formulas (i.e. provable equivalence according to $\vdash$) is a congruence on the algebra of formulas, then we can form the {\em Lindenbaum-Tarski algebra} as the quotient of provably equivalent formulas of the algebra of formulas. If the quasivariety $\vv Q_{\mathcal{L}}$ generated by the Lindembaum-Tarski algebra satisfies further conditions, then it is the class of algebraic models of $\mathcal{L}$ and $\mathcal{}$ is {\bf algebraizable} with {\bf equivalent algebraic semantics} $\vv Q_{\mathcal{L}}$; Blok and Pigozzi \cite{BlokPigozzi1989} formalized this connection and gave necessary and sufficient conditions for a logic to be algebraizable.  Essentially, one needs a finite set of equations $\tau =\{\d_i \app \e_i: i = 1,\dots,n\}$ in the language of $\vv Q_{\mathcal{L}}$  and a finite set of formulas of $\mathcal L$
in two variables $\Delta(x,y)=\{\f_1(x,y),\dots,\f_m(x,y)\}$ that allow to transform equations, quasiequations and universal sentences in $\vv Q_{\mathcal{L}}$ into formulas, rules and clauses of $\mathcal{L}$ and viceversa; moreover this transformation must respect both the consequence relation and the semantical consequence. That is to say, for all sets of formulas $\Gamma$ of $\cc L$ and formulas $\f \in \alg T_{\tau}(\omega)$
$$
\Gamma \vdash \f\quad\text{if and only if}\quad \{\d_i(\Gamma) \app \e_i(\Gamma): i=1,\dots,n\} \vDash_{\vv Q_{\cc L}} \{\d_i(\f) \app \e_i(\f): i=1,\dots,n\}
$$
where of course $\d_i(\Gamma) \app \e_i(\Gamma)$ is a shorthand for   $\d_i(\psi) \app \e_i(\psi)$ for all $\psi \in \Gamma$ and also
$$
\vv Q_{\cc L} \vDash (x \app y) \Leftrightarrow (\bigwedge_{i=1}^n \bigwedge_{j=1}^m (\d_i(\f_j(x,y)) \app \e_i(\f_j(x,y))).
$$
A quasivariety $\vv Q$ is a {\em quasivariety of logic} if it is the equivalent quasivariety semantics for some logic $\mathsf{L}_\vv Q$; the Galois connection between algebraizable logics and quasivarieties of logic is given by
$$
\mathcal{L}_{\vv Q_{\mathcal{L}}} = \mathcal{L} \qquad\qquad \vv Q_{\mathcal{L}_\vv Q} = \vv Q.
$$

\subsection{Structural completeness in algebra and logic}

An {\bf extension} of $\mathcal{L}$ over the language $\tau$ is a logic $\mathcal{L}'$ over the same language such that $\Sigma \vdash_\mathcal{L} \d$ implies $\Sigma \vdash_{\mathcal{L}'} \d$. A {\bf finitary extension} of $\mathcal{L}$ is an extension of $\mathcal{L}$  obtained by adding a set of rules to the calculus of $\mathcal{L}$; an {\bf axiomatic extension} of $\mathcal{L}$ is an extension of $\mathcal{L}$ obtained by adding a set of axioms to the calculus of $\mathcal{L}$. The class of the (finitary, axiomatic) extension of the logic $\mathcal{L}$ is never empty, since $\mathcal{L}$ clearly belongs to it.
Therefore we can define closure operators in a standard way and hence complete lattices; so $\op{Th}(\mathcal{L})$, $\op{Th}_{f}(\mathcal {L})$ and $\op{Th}_a(\mathcal{L})$ will be the lattice of all the extensions, finitary extensions and axiomatic extensions of $\mathcal{L}$ respectively.
Via the Blok-Pigozzi connection we get:

\begin{theorem}\label{main1} Let $\mathcal{L}$ be an algebraizable logic with equivalent algebraic semantics $\vv Q_\mathcal {L}$. Then
\begin{enumerate}
\item  the lattice $\op{Th}_a(\mathcal L)$ of the axiomatic extensions of $\mathcal{L}$ is dually isomorphic with $\Lambda(\vv Q)$;
\item the lattice $\op{Th}_f(\mathcal L)$ of the finitary extensions of $\mathcal{L}$ is dually isomorphic with $\Lambda_q(\vv Q)$.
\end{enumerate}
\end{theorem}

An algebraizable logic $\mathcal L$ is {\bf tabular} if it is the logic of a finite algebra; in other words $\vv Q_\mathcal L$ is a finitely generated quasivariety, i.e. $\vv Q_\mathcal L = \QQ(\alg A)$ for some finite algebra $\alg A$. A logic is {\bf locally tabular} if, for any finite $k$, there exist only finitely many
pairwise nonequivalent formulas in $\mathcal L$ built from the variables $\vuc xk$. It is clear that an algebraizable logic is locally tabular if and only if
$\vv Q_\mathcal L$ is a locally finite quasivariety.

A clause $\Sigma \Rightarrow \Delta$ is {\bf admissible} in $\mathcal{L}$ if every substitution that makes the premises into a theorem of $\mathcal L$, also makes at least one of the conclusions in $\Delta$ a theorem of $\mathcal{L}$. In particular a rule is admissible in $\mathcal{L}$  if, when added to its calculus, it does not produce any new theorem. A clause $\Sigma \Rightarrow \Delta$ is {\bf derivable} in $\mathcal{L}$ if $\Sigma \vdash \d$ for some $\d \in \Delta$.
A logic $\mathcal {L}$ is {\bf structurally complete} if every admissible rule of $\mathcal{L}$ is derivable in $\mathcal{L}$; a logic is {\bf hereditarily structurally complete} if every finitary extension of $\mathcal{L}$ is structurally complete. Determining structural completeness of a logic is in general a very deep and challenging problem; here we will use only the parts of the theory that are necessary but for an extensive treatment of the subject we direct the reader to \cite{AglianoUgolini2023}.

A quasivariety $\vv Q$ is {\bf structural} if for every subvariety $\vv Q'\sse \vv Q$, $\HH(\vv Q') = \HH(\vv Q)$ implies $\vv Q'= \vv Q$.

\begin{theorem} \label{structural} \cite{Bergman1991}For a quasivariety $\vv Q$ the following are equivalent:
\begin{enumerate}
\item $\vv Q$ is structural;
\item $\QQ(\alg F_\vv Q(\o) = \vv Q$.
\end{enumerate}
\end{theorem}

For any quasivariety $\vv Q$, we define the {\bf structural core of  $\vv Q$}   as the smallest $\vv Q'\sse \vv Q$ such that $\HH(\vv Q) = \HH(\vv Q')$. The structural core
of a quasivariety always exists:

\begin{corollary} For any quasivariety $\vv Q$, $\QQ(\vv F_\vv Q(\o))$ is structural and  it is the structural core of  $\vv Q$.
\end{corollary}
\begin{proof} $\QQ(\vv F_\vv Q(\o))$ is structural by Theorem \ref{structural}; if $\vv Q' \sse \vv Q$ is such that $\HH(\vv Q') = \HH(\vv Q)$, then clearly $\alg F_\vv Q(\o) \in \vv Q'$ from which the thesis follows.
\end{proof}

It follows at once that a quasivariety $\vv Q$ is structural if and only if it coincides with its structural core. As a consequence the structural subquasivarieties of a quasivariety $\vv Q$ are exactly those that coincide with the structural cores of  $\vv Q'$ for some $\vv Q'\sse \vv Q$; even more, since $\HH(\vv Q)$ is a variety, the structural subquasivarieties of a variety $\vv V$ are exactly the structural cores of $\vv V'$ for some subvariety $\vv V'$ of $\vv V$. As we will see in the sequel this observation is particularly useful when the free countably generated algebra in $\vv V$ has a reasonable description.

If $\vv Q$ is a quasivariety and $\vv Q'$ is a subquasivariety of $\vv Q$ we say that $\vv Q'$ is {\bf equational} in $\vv Q$ if $\vv Q' = \HH(\vv Q') \cap \vv Q$; this is clearly equivalent to saying that $\vv Q'$ is axiomatized modulo $\vv Q$ by a set of equations. A quasivariety $\vv Q$ is {\bf primitive} if each subquasivariety of $\vv Q'$ is equational in $\vv Q$. The following lemma is straightforward:

\begin{lemma} For a quasivariety $\vv Q$ the following are equivalent;
\begin{enumerate}
\item $\vv Q$ is primitive;
\item every subquasivariety of $\vv Q$ is structural.
\end{enumerate}
\end{lemma}

For locally finite quasivarieties we have a necessary and sufficient condition due essentially to Gorbunov; an algebra $\alg A$ is {\bf weakly projective} in a class $\vv K$ if for all $\alg B \in \vv K$, $\alg A \in \HH(\alg B)$ implies $\alg A \in \II\SU(\alg B)$. An algebra is {\bf projective} in $\vv K$ if it is weakly projective in $\vv K$ and the epimorphism and the embedding witnessing weak projectivity compose to the identity on $\alg A$.

\begin{theorem}\label{maingorbunov}(see \cite{Gorbunov1998}) For a locally finite quasivariety $\vv Q$ the following are equivalent:
\begin{enumerate}
\item $\vv Q$ is primitive;
\item every finite $\vv Q$-irreducible algebra is weakly projective in $\vv Q$;
\item every finite $\vv Q$-irreducible algebra is weakly projective in the class of finite algebras in $\vv Q$.
\end{enumerate}
\end{theorem}

From the Blok-Pigozzi connection we get at once:

\begin{theorem} \label{main2} Let $\mathcal{L}$ be an algebraizable logic with equivalent algebraic semantics $\vv Q_\mathcal {L}$. Then
\begin{enumerate}
\item $\mathcal L$ is structurally complete if and only if $\vv Q_\mathcal L$ is structural;
\item $\mathcal L$ is hereditarily structurally complete if and only if $\vv Q_\mathcal L$ is primitive.
\end{enumerate}
\end{theorem}

Combining Theorems \ref{main1} and \ref{main2} it follows at once that if $\mathcal L$ is algebraizable and $\vv Q_\mathcal L$ is primitive, then every finitary extension of $\mathcal {L}$ is axiomatic.

Some investigations using (part of) the machinery we have described in this section has already been used to investigate structural completeness in (quasi)varieties of fuzzy logics (see for instance  \cite{OlsonRafVanAlten2008} or \cite{CintulaMetcalfe2009}). In this note however we will get more into
the details for a specific variety and this will allow us to characterize all the structurally complete finitary extensions of positive {\L}ukasiewicz logic.

\subsection{Hoops as quasivarieties of logics}

A commutative integral residuated lattice is an algebra\\ $\la A,\join,\meet, \cdot, \imp,1\ra$ such that
\begin{enumerate}
\item $\la A,\join,\meet, 1 \ra$ is a lattice with largest element $1$;
\item $\la A,\cdot,1\ra$ is a commutative monoid;
\item $(\cdot,\imp)$ form a residuated pair w.r.t. the lattice ordering, i.e. for all $a,b,c \in A$
$$
a \cdot b \le c\qquad\text{if and only if}\qquad a \le b \imp c.
$$
\end{enumerate}
In what follows, we will often write $xy$ for $x \cdot y$. If we augment the signature with an extra constant $0$  that is the least element in the lattice order, then we get $\mathsf{FL}_{ew}$-algebras. Commutative integral residuated lattices and $\mathsf{FL}_{ew}$-algebras form varieties that have a very rich structure; for an equational axiomatization and a list of valid identities we refer the reader to \cite{BlountTsinakis2003}.

Varieties of $\mathsf{FL}_{ew}$-algebras and commutative integral residuated lattices are {\em ideal determined} (w.r.t. $1$) in the sense of  \cite{AglianoUrsini1992}; this means that there is a one-to-one correspondence
(which is in fact a lattice isomorphism) between the congruences of an algebra $\alg A$ and certain special subsets of $\alg A$. In the present case if $\alg A$ is an $\mathsf{FL}_{ew}$-algebra or a commutative integral residuated lattice a {\bf filter} of $\alg A$ is a filter $F$ of the lattice structure which is also closed under the monoidal operation.  If $\th \in \Con A$ then $1/\th$ is clearly a filter of $\alg A$ and it is easily checked that
if $F$ is a filter then  $\th_F =\{(a,b):  a \imp b,b \imp a \in F\} \in \Con A$ and the correspondence is
$$
\th \longmapsto 1 /\th  \qquad  F \longmapsto \th_F.
$$
If $F$ is a filter we will write $\alg A/F$ for $\alg A/\th_F$.

 Let's focus on two equations that bear interesting consequences, i.e., prelinearity and divisibility:
\begin{align}
&(x \imp y) \join (y \imp x) \app 1.\tag{prel}\\
&x(x \imp y) \app y(y \imp x); \tag{div}
\end{align}
It can be shown (see \cite{BlountTsinakis2003} and \cite{JipsenTsinakis2002}) that a subvariety of $\mathsf{FL_{ew}}$ satisfies the prelinearity equation (prel) if and only if any algebra therein is a subdirect product of totally ordered algebras, and this implies via Birkhoff's Theorem that all the subdirectly irreducible algebras are totally ordered. Such varieties are called {\em representable} (or {\em semilinear}) and the subvariety axiomatized by (prel) is the largest subvariety of $\mathsf{FL_{ew}}$ that is representable; such variety is usually denoted by $\MTL$, since it is the equivalent algebraic semantics of Esteva-Godo's {\em Monoidal t-norm based logic} \cite{EstevaGodo2001}.

If an algebra in $\mathsf{FL_{ew}}$ satisfies both (prel) and (div) then it is called a $\BL$-algebra and the variety of all $\BL$-algebras is denoted by $\BL$. Again the name comes from logic: the variety of $\BL$-algebras is the equivalent algebraic semantics of {\em H\'ajek's Basic Logic} $\mathcal{BL}$ \cite{Hajek1998}. A systematic investigation of varieties of $\BL$-algebras started with \cite{AglianoMontagna2003} and it is still ongoing (see \cite{Agliano2017c} and the bibliography therein).

It follows from the definition that given a variety of bounded commutative integral residuated lattices, the class of its {\em $0$-free subreducts} is a class of residuated lattices; we have a very general result.

\begin{lemma} Let $\vv V$ be any subvariety of $\mathsf{FL_{ew}}$; then the class $\SU^0(\vv V)$ of the zero-free subreducts of algebras in $\vv V$ is a variety.
\end{lemma}
\begin{proof}The proof is as in Proposition 1.10 of \cite{AFM}; it is stated for varieties of $\BL$-algebras but it uses only the description of the congruence filters, that can be used in any subvariety of $\mathsf{FL_{ew}}$ (as the reader can easily check).
\end{proof}

This implies at once that if a variety of $\mathsf{FL}_{ew}$-algebras is the equivalent algebraic semantics of a logic $\mathcal{L}$, then the variety of its zero-free subreducts is the equivalent algebraic semantics of the positive fragment $\mathcal{L}^+$. A {\bf basic hoop} is a zero-free subreduct of a divisible and prelinear $\mathsf{FL}_{ew}$-algebra. Note that in any $\mathsf{FL}_{ew}$-algebras
the prelinearity equation makes the join definable using $\meet$ and $\imp$ (see for instance \cite{Agliano2018c}):
\begin{equation*}
((x \imp y) \imp y) \meet ((y \imp x) \imp x) \app x \join y.
\end{equation*}
So basic hoops are often presented in the signature $\meet,\imp, 1$;  as noted in \cite{AFM} the variety $\mathsf{BH}$ of basic hoops is the equivalent algebraic semantics of the positive fragment of the logic $\mathcal{BL}$.

A {\bf Wajsberg hoop} is a basic hoop satisfying the so-called {\em Tanaka's equation}
\begin{equation*}
(x \imp y) \imp y \app (y \imp x) \imp x.
\end{equation*}
If we add a constant $0$ to the signature that is the least element in the lattice order then we have {\bf Wajsberg algebras}; Wajsberg algebras are term equivalent to $\mathsf{MV}$-algebras (see \cite{AglianoPanti1999} p. 354 for a detailed explanation) and the variety of $\mathsf{MV}$-algebras is usually presented as the equivalent algebraic semantics of \L ukasiewicz logic $\mathcal{MV}$. It follows that the variety $\mathsf{WH}$ of Wajsberg hoops, which is the variety of zero-free subreducts of Wajsberg algebras, is the equivalent algebraic semantics of $\mathcal{MV}^+$, i.e. the positive fragment of \L ukasiewicz logic.

In $\mathsf{FL}_{ew}$-algebras it is customary to introduce the derived operation $\neg x := x \imp 0$; now in a bounded (i.e. with a minimum element $a$) Wajsberg hoop we can still introduce a negation $\neg x = x \imp a$ that is of course not a term but rather a polynomial. Now it is easy to verify that any bounded Wajsberg hoop is polynomially equivalent to a Wajsberg algebra and we will freely use the expression $\neg x$ letting the context clear the meaning.

A commutative integral residuated lattice is {\bf cancellative} if the underlying monoid is cancellative in the usual sense.

\begin{lemma} \cite{BlokFerr2000} Every cancellative basic hoop is a Wajsberg hoop. A totally ordered Wajsberg hoop is either cancellative or bounded.
\end{lemma}

This allows us to show that the connection between Wajsberg hoops and Wajsberg algebras is even stricter. For instance the operator $\II\SU\PP_u$ on Wajsberg hoops has been studied in \cite{AglianoMontagna2003} using the results about  Wajsberg algebras appeared in \cite{Gispert2002}; while we maintain that it should be clear why we can do this (and in \cite{AglianoMontagna2003} no explanation was given), maybe some clarification is useful.  Wajsberg algebras are polynomially equivalent to bounded Wajsberg hoops; it is easy to see that if $\OO$ is a class operator that is a composition of $\II,\HH,\SU,\PP,\PP_u$,
$\alg A, \alg B$ are Wajsberg algebras and $\alg A_0,\alg B_0$ are their Wajsberg hoop reducts, then $\OO(\alg A) \sse \OO(\alg B)$ if and only if $\OO(\alg A_0) = \OO(\alg B_0)$. This allows us to consider bounded Wajsberg hoops {\em as they were} Wajsberg algebras. Since a totally ordered Wajsberg hoop is either bounded or cancellative we can use results about Wajsberg algebras and integrate them with the cancellative case.

\section{Some useful tools}

\subsection{Wajsberg chains} Bounded Wajsberg hoops have a {\em canonical representation}. Let $\alg G$ be a lattice ordered abelian group; by \cite{Mun1986}, if $u$ is a strong unit of $\alg G$ we can construct a bounded Wajsberg hoop $\Gamma(\alg G,u) = \la [0,u], \imp,\cdot, 0, u\ra$ where $ab = \op{max}\{a+b-u,0\}$ and $a \imp b = \op{min}\{u-a+b,u\}$. The main result of \cite{Mun1986} is that any bounded Wajsberg hoop can be presented in this way (really there is a categorical equivalence between the category abelian $\ell$-groups with strong unit and the category of bounded Wajsberg hoops). Let now  $\mathbb Z \lextimes \mathbb Z$ denote the lexicographic product of two copies of $\mathbb Z$. In other
words, the universe is the Cartesian product and the group operations are defined componentwise
and the ordering is the lexicographic ordering (w.r.t. the natural ordering of $\mathbb Z$); then  $\mathbb Z \lextimes \mathbb Z$ is a totally ordered abelian group and we can apply  $\Gamma$ to it. A {\bf Wajsberg chain} is a totally ordered Wajsberg hoop. Let's define some useful Wajsberg chains:
\begin{enumerate}
\itemb the finite Wajsberg chain with $n+1$ elements $\alg{\L}_n = \Gamma(\mathbb Z,n)$;
\itemb the infinite finitely generated Wajsberg chain $\alg{\L}_n^\infty = \Gamma(\mathbb Z\lextimes\mathbb Z, (n,0))$;
\itemb the infinite  finitely generated  Wajsberg chain $\alg {\L}_{n,k}=\Gamma(\mathbb Z\lextimes\mathbb Z, (n,k))$;
\itemb the unbounded Wajsberg chain $\alg C_\omega$; this can be regarded either as the free monoid on one generator, where the product is the monoid product and $a^l \imp a^m = a^{\op{max}(l-m,0)}$ or, equivalently, as the negative cone of $\mathbb Z$ with the operations defined in the obvious way.
\end{enumerate}

We observe that $\alg {\L}_n^\infty = \alg {\L}_{n,0}$;
moreover the proof of the following is a simple verification:

\begin{lemma}\label{cond1}  \begin{enumerate}
\item For $n,m \in \mathbb N$, $\alg {\L}_n \in \II\SU(\alg {\L}_m)$ if and only if $\alg {\L}_n \in \II\SU(\alg {\L}_m^\infty)$ if and only if $n \mathrel{|} m$.
\item For $n,r,j \in \mathbb N$, $\alg{\L}_n \in \II\SU(\alg{\L}_{r,j})$ if and only if $n \mathrel{|} \gcd\{r,j\}$.
\item If $\alg A$ is a cancellative Wajsberg chain and $a \in A\setminus \{1\}$, then $a$ generates a subalgebra of $\alg A$ isomorphic with $\alg C_\o$.
\end{enumerate}
\end{lemma}

In \cite{AglianoPanti1999} it has been shown that every proper variety of Wajsberg hoops is generated by a finite number of chains of the type described above.
More precisely a {\bf presentation} $P$ is a triple  $(I,J,K)$  $I,J$ are finite subsets of $\mathbb N\setminus\{0\}$ and $K \sse \{\o\}$. We say that the presentation
$P =(I,J,K)$ is {\bf reduced} if
\begin{enumerate}
\ib $I \cup J \cup K \ne \emptyset$;
\ib if $J \ne \emptyset$ then $K = \emptyset$;
\ib no $m \in I$ divides any $m' \in (I \setminus \{m\}) \cup J$;
\ib no $n \in J$ divides any $n'\in J\setminus\{m\}$.
\end{enumerate}
For any reduced presentation $P=(I,J,K)$ we define a set  of Wajsberg hoops $\vv K_P$ in the following way:
\begin{enumerate}
\ib if $P= (I,J,\emptyset)$ then $\vv K_P =\{\alg {\L}_i: i \in I\} \cup \{ \alg {\L}_j^\infty:  j\in J\}$,
\ib if $P =(I,\emptyset,\{\o\})$ then $\vv K_P  = \{\alg {\L}_i: i \in I\} \cup \{\alg C_\o\}$.
\end{enumerate}
Then we set $\vv V(P)= \VV(\vv K_P)$ and $\vv Q(P) = \QQ(\vv K_P)$.

\begin{theorem}\label{aglianopanti} \cite{AglianoPanti1999} The proper subvarieties of Wajsberg hoops are in one-to-one correspondence with the reduced presentations via
$$
P \longmapsto \vv V(P).
$$
\end{theorem}

\subsection{The construction of $\alg B_\Delta$}\label{bdelta}

Any proper subvariety of Wajsberg hoops is axiomatizable (modulo Wajsberg hoops) by an equation in a single variable; this is essentially Theorem 4.4 in \cite{AglianoPanti1999} and its proof uses the functional characterization of
free Wajsberg hoops which we will be using later in this paper. From that, it follows that for any proper subvariety $\vv V$ of Wajsberg hoops $\QQ(\alg F_\vv V(\o)) = \QQ(\alg F_\vv V(x))$. Indeed, $\vv V=\HH\QQ(\alg F_\vv V(\o))=\HH\QQ(\alg F_\vv V(x))$ because $\vv V$ can be axiomatized in one variable and, since $\QQ(\alg F_\vv V(\o))$ is structural, this means that it has to be the smallest quasivariety that generates $\vv V$, so $\QQ(\alg F_\vv V(\o))\subseteq \QQ(\alg F_\vv V(x))$; the other inclusion is trivial.

In this section we will give an alternative description of $\alg F_\vv V(x)$ where $\vv V$ is a proper subvariety of Wajsberg hoops.

\begin{lemma}\label{onegenerated}
Let $\alg{A}$ be a totally ordered Wajsberg hoop, assume that $\alg{A}$ is one-generated, then $\alg{A}\in \VV(\alg{\L}_n^\infty)$ if and only if one of the following holds:
\begin{enumerate}
    \item there exists $1\leq k|n$ such that $\alg{A}\cong\alg{L}_k$;
    \item there exist $1\leq k|n$ and $0\leq h<k$ with $k,h$ relatively prime such that $\alg{A}\cong\alg{\L}_{k,h}$;
    \item $\alg{A}\cong C_\omega$.
\end{enumerate}
\end{lemma}
\begin{proof} As $\alg{A}$ is totally ordered, it is either bounded or cancellative. If it is cancellative then since it is one-generated, it must be $\alg{A}\cong C_\omega$; if it is bounded, then it can be seen as an $\mathsf{MV}$-algebra, so, by \cite{DiNolaGrigoliaPanti1998} (Theorem 1.8), either 1. or 2. holds.

Conversely if $\alg A \cong C_\omega$, then clearly $\alg A \in \VV(\alg{\L}_n^\infty)$; in the other cases we appeal again to Theorem 1.8 in \cite{DiNolaGrigoliaPanti1998}.
\end{proof}

From now on, given a finite subset $X$ of $\mathbb{N}$, we will denote by $X{\downarrow}$ the set of all the divisors of elements of $X$.

\begin{lemma}\label{onegenerated2} Let $\alg A$ be a totally ordered Wajsberg hoop, assume that $\alg A$ is one generated and let $P = (I,J,\emptyset)$ be a reduced presentation with $J \ne \emptyset$.  Then $\alg A\in \vv V(P)$  if and only if one of the following holds:
\begin{enumerate}
    \item there exists $k\in I{\downarrow}\cup J{\downarrow}$ such that $\alg A\cong\alg{\L}_{k}$;
    \item there exist $k\in J{\downarrow}$ and $0\leq h<k$ with $k,h$ relatively prime such that $\alg A\cong\alg{\L}_{k,h}$;
    \item $\alg A\cong C_\omega$.
\end{enumerate}
\end{lemma}
\begin{proof} The ``only if'' part is exactly as in Lemma \ref{onegenerated}.

If $\alg A\cong\alg{\L}_k$ for some $k\in I{\downarrow}\cup J{\downarrow}$, then, if $k\in I{\downarrow}$ there exists $i\in I$ such that $\alg A\in \VV(\alg{\L}_i)\subseteq \vv V(I,J,\emptyset)$; if $k\in J{\downarrow}$, then by Lemma \ref{onegenerated} there exists a $j\in J$ such that $\alg A\in \VV(\alg{\L}_j^\infty)\subseteq \vv V(I,J,\emptyset)$.

If $\alg A\cong\alg{\L}_{k,h}$ for some $k\in J{\downarrow}$ and $0\leq h<k$ with $h,k$ relatively prime, then by Lemma \ref{onegenerated} there exists $j\in J$ such that $\alg A\in \VV(\alg{\L}_j^\infty)\subseteq \vv V(I,J,\emptyset)$.
Finally, if $\alg A\cong C^\omega$ clearly $\alg A\in \VV(\alg{\L}_j^\infty)$ for every $j\in J{\downarrow}$, so, since $J\neq\emptyset$, $\alg A\in \vv V(I,J,\emptyset)$.
\end{proof}

\begin{remark}\label{rem:generators}
If $\alg A\cong\alg{\L}_k$, then $h$ generates $\alg A$ if and only if $h,k$ are relatively prime.

If $\alg A\cong\alg{\L}_{k,h}$ with $k\neq 1$ and $h<k$ then there exists a unique $g_{k,h}\in\alg A$ with $g_{k,h}\leq\neg g_{k,h}$  such that $a$ generates $\alg{A}$ if and only if $a=g_{k,h}$ or $a=\neg g_{k,h}$. Moreover $g_{k,1}=(1,0), g_{k,k-1}=(1,1)$ and, if $h\neq 1, h\neq k-1$, then $g_{k,h}=(r,s)$ with $1<r<\frac{k}{2}$. If $k=1$ we get that $\alg{\L}_{1,0}$ is generated by $g_{1,0}=(0,1)$, but this time $\neg g_{1,0}=(1,-1)$ generates a subalgebra of $\alg{\L}_{1,0}$ isomorphic to $C_\omega$.

Note that in all these cases, we can use the operation $\neg$ because all the algebras are bounded; in particular, since the generator of $\alg{\L}_{k,h}$ has always order $2$, we can write $0$ as $g_{k,h}^2$, so for every $a\in\alg{\L}_{k,h}$ we get $\neg a=a\rightarrow g_{k,h}^2$.
\end{remark}

Now we fix a  reduced triple $P=(I,J,K)$ and let
\begin{align*}
&\Delta_I = \{(k,h,2): 0 \le h < k \in I{\downarrow}, \text{$k,h$ relatively prime}\}\\
&\Delta_J = \{(k,h,i): i \in \{0,1\},  h < k \in J{\downarrow}, \text{$k,h$ relatively prime}\}\\
&\Delta_K = \left\{
              \begin{array}{ll}
                \emptyset, & \hbox{if $J \ne \emptyset$;} \\
                \{(0,0,3)\}, & \hbox{otherwise.}
              \end{array}
            \right.
\end{align*}
Moreover for any $k,h$ we define
\begin{align*}
&\alg A^0_{k,h} = \alg A^1_{k,h} := \alg {\L}_{k,h}\\
&\alg A^2_{k,h} = \alg {\L}_k\\
&\alg A^3_{k,h} = \alg C_\o.
\end{align*}
If $\Delta = \Delta_I \cup \Delta_J \cup \Delta_K$ we let $\alg A_\Delta = \prod_{(k,h,i) \in \Delta} \alg A^i_{k,h}$; moreover we denote by $c$ the generator of $\alg C_\o$.
We want to define a $\ol{g} \in \alg A_\Delta$ by cases; for any $(k,h,i) \in \Delta$
\begin{align*}
&\text{if $J=\emptyset$,}\ \ \ol{g}(k,h,i) = \left\{
                                           \begin{array}{ll}
                                             h, & \hbox{if $i=2$;} \\
                                             c, & \hbox{if $i=3$.}
                                           \end{array}
                                         \right.\\
&\text{if $J \ne\emptyset$,}\ \ \ol{g}(k,h,i) =  \left\{
                                              \begin{array}{ll}
                                                 g_{k,h}, & \hbox{if $i=0$;} \\
                                                 \neg g_{k,h}, & \hbox{if $i=1$;} \\
                                                 h, & \hbox{if $i=2$.}
                                               \end{array}
                                             \right.
\end{align*}

\begin{theorem}\label{thm:bdelta} Let $P=(I,J,K)$ be any reduced triple and let $\ol g$ and $\alg A_\Delta$ as above. If $\alg B_\Delta$ is the subalgebra of $\alg A$ generated by
$\ol g$, then $\alg B_\Delta \cong \alg F_{\vv V(P)}(x)$.
\end{theorem}
\begin{proof} First we observe that $\alg A_\Delta \in \vv V(P)$ and so does $\alg B_\Delta$.  Suppose that $p(x) \app q(x)$ is an equation that fails
in $\vv V(P)$; then it must fail in some one-generated totally ordered algebra $\alg C \in \vv V(P)$ and such algebra is either bounded or cancellative.

First, let us show that we only need to discuss the case in which $p(x)\app q(x)$ fails in a generator of $\alg C$. Suppose that the equation fails in some $x\in\alg C$, then, if we call $\alg C'$ the subalgebra of $\alg C$ generated by $x$, we have that $p(x)\app q(x)$ fails in the generator of $\alg C'$, that is still an algebra in $\vv V(P)$.

Suppose that $J = \emptyset$.  If $\alg C$ is bounded, then it cannot be infinite, as $\alg {\L}_n^\infty \notin \vv V(P)$ for any $n \in \mathbb N$. Hence it must be equal to $\alg {\L}_k$
for some $k \in I{\downarrow}$; this implies that $p(x) \app q(x)$ fails in $g(k,h,2)$ for some $h$ and, as above, fails in $\alg B_\Delta$. This covers the case $K=\emptyset$ and half of the case $K \ne \emptyset$. To conclude, if $\alg C$ is cancellative, then the equation must fail in $\alg C_\o$ and hence (if $c$ is the generator of $\alg C_\o$),
$p(c) \ne q(c)$; this implies that $p(\ol{g}(0,0,3))\ne p(\ol{g}(0,0,3))$, so $p(\ol g) \ne q(\ol g)$ and $p(x) \app q(x)$ fails in $\alg B_\Delta$.

Suppose now that $J \ne \emptyset$ (and thus $K = \emptyset$). By Lemma \ref{onegenerated2} we have only three possibilities.

If $\alg C \cong \alg {\L}_k$ for some $k \in I{\downarrow} \cup J\downarrow$; if $k \in I{\downarrow}$ then $p(x) \app q(x)$ fails in some generator of
$\alg {\L}_k$, so it fails in some $\ol{g}(k,h,2)$ for some $h$ and eventually fails in $\alg B_\Delta$. If $k \in J{\downarrow}$ again $p(x) \app q(x)$ fails
in some generator $h$ of $\alg {\L}_k$; now $h$ and $k$ must be relatively prime and thus $p(x) \app q(x)$ fails in $\alg {\L}_{k,h}$.
By Remark \ref{rem:generators} it fails either in $g_{h,k}$ or $\neg g_{h,k}$, hence $p(x) \app q(x)$ fails either in $\ol{g}(h,k,0)$ or
$\ol{g}(h,k,1)$.  In any case $p(x) \app q(x)$  fails in $\alg B_\Delta$.

If $\alg C \cong \alg{\L}_{h,k}$ the argument is similar to the one above, but easier. Finally if $\alg C\cong \alg C_\o$, since $\ol{g}(1,0,1)=\neg g_{1,0}=(1,-1)$ generates a subalgebra of $\alg{\L}_{1,0}$ isomorphic to $C^\omega$, for sure $p(\ol{g}(1,0,1)) \ne q(\ol{g}(1,0,1)) $ and so again $p(x)\app q(x)$   fails in $\alg B$.

We have thus proved that every equation in one variable that fails in $\vv V(P)$ fails in $\alg B_\Delta$. At last, let us show that this is sufficient to say that $\alg B_\Delta \cong \alg F_{\vv V(P)}(x)$. Indeed, since $\alg B_\Delta$ is in $\vv V(P)$ and it is one-generated, we know that we have a surjective homomorphism $\varphi$ from $\alg F_{\vv V(P)}(x)$ to $\alg B_\Delta$; suppose now that $\varphi$ is not injective, this means that Ker($\varphi$) is non-trivial, so there exist two terms $p,q\in\alg F_{\vv V(P)}(x)$ such that $p\neq q$ but $\varphi(p)=\varphi(q)$, thus $p(x)\not\app q(x)$ in $\alg F_{\vv V(P)}(x)$ but $p(x)\app q(x)$ in $\alg B_\Delta$. Hence $\varphi$ must be an isomorphism.
\end{proof}

\subsection{Wajsberg functions}

A \textbf{McNaughton function} over the $n$-cube is a continuous function $f:[0,1]^n\to[0,1]$ such that there exist finitely many linear functions $f_1,\dots,f_k$, where each $f_i$ is of the form $f_i=a_i^0 x_0+a_i^1 x_1+\dots+a_i^n x_n+b_i$ with $a_i^0\dots a_i^n,b_i$ integers, and such that for any $v\in[0,1]^n$ there exists $i\in\{1,\dots k\}$ with $f(v)=f_i(v)$.  A McNaughton function $f(\vuc xn)$ is a {\bf Wajsberg function} if
$f(1,1,\dots,1) = 1$.

\begin{theorem}\label{Wajsberg functions} \cite{AglianoPanti1999} For each $n$, the free $n$-generated Wajsberg hoop $\alg F_{\WH}(n)$ is isomorphic to the algebra of all Wajsberg functions over the $n$-cube, where the operations are defined pointwise.
\end{theorem}

So we can always identify an  $n$-ary term in the language of Wajsberg hoops with a  Wajsberg function over the $n$-cube. Conversely, given a Wajsberg function over the $n$-cube, we can associate to it an equivalence class of Wajsberg terms (where the equivalence is of course mutual provability in the theory).
With the usual abuse of notation we will identify the class with any of its representatives, i.e. given a Wajsberg function $f$ we will denote by $\widehat f$ the Wajsberg term which is a representative to the equivalence class corresponding to $f$.

In \cite{AglianoPanti1999} the authors used this representation to give an easy way to axiomatize all proper subvarieties of Wajsberg hoops. Let $(I,J,K)$ be a reduced triple, we define two finite subsets $\mathcal{I},\mathcal{J}$ of rational points of $[0,1]$ as
\begin{enumerate}
    \ib if $K\neq\emptyset$ then $\mathcal{J}=\{1\}$;
    \ib if $K=\emptyset$ then $\mathcal{J}=\{v\in[0,1] : \text{den}(v)\in J{\downarrow}\}$;
    \ib $\mathcal{I}=\{u\in[0,1] : \text{den}(u)\in I{\downarrow}\}\backslash\mathcal{J}$.
\end{enumerate}
Given a reduced triple $(I,J,K)$ an {\bf $(I,J,K)$-comb} is any $\alpha\in\alg F_{\WH}(x)$ such that
\begin{enumerate}
    \item for every $v\in\mathcal{J}$, there exists a neighborhood $V$ of $v$ such that $\alpha=1$ on $V$;
    \item for every $u\in\mathcal{I}$, $\alpha(u)=1$;
    \item for every $u\in\mathcal{I}$ there exists $v\in\mathcal{I}$ such that den($v$)$|$den($u$) and $\alpha$ is not identically $1$ on any neighborhood of $v$;
    \item if $d\notin (I\cup J){\downarrow}$, then there exists $0\leq h<d$ with $\alpha(h/d)\neq 1$.
\end{enumerate}

\begin{theorem}\label{combs}\cite{AglianoPanti1999}
Let $P=(I,J,K)$ be a reduced triple and let $\alpha(x)\in\alg F_{\WH}(x)$. Then the identity $\alpha(x)=1$ axiomatizes $\vv V(P)$ relative to $\WH$ if and only if $\alpha$ is an $(I,J,K)$-comb.
\end{theorem}

This is a very powerful result, in that it gives a procedure that allows to axiomatize every proper
subvariety of Wajsberg hoops, and combs are quite easy to construct.
Next, we have a very useful lemma.

\begin{lemma}\label{identities} Let $p(x) \app q(x)$ be an identity in the language of Wajsberg hoops and let $f,g$ be Wajsberg functions such that
$p = \widehat f$ and $q = \widehat g$. Then for any $n,k \in \mathbb N$ with $ k \le n$
\begin{enumerate}
\item if $f(\frac{k}{n})= g(\frac{k}{n})$, then $p(k) = q(k)$ where $k \in \alg {\L}_n$;
\item if $f(x) = g(x)$ in a neighborhood of $1$, then $\alg C_\o \vDash p(x)\app q(x)$;
\item if $f(x) = g(x)$ in a neighborhood of $\frac{k}{n}$, then $p(c) =q(c)$ for any $c \in\alg {\L}_{n,h}$ such that $c/\op{Rad}(\alg {\L}_{n,h}) = k$.
\end{enumerate}
\end{lemma}

The proof of Lemma \ref{identities} can be extracted from the proof of Theorem 3.3 in \cite{AglianoPanti1999}, by setting $\kappa =1$.

Next, we want to have an easy way to construct Wajsberg functions and force them to have certain fixed values (see \cite{Jerabek2010}).  If $0=t_0<t_1<\dots<t_k=1$ and $x_0,\dots,x_k\in[0,1]$, then we denote by $f=L(t_0,x_0;\dots;t_k,x_k)$ the continuous picewise linear function $f:[0,1]\to[0,1]$ such that $f(t_i)=x_i$ and $f$ is linear on each interval $[t_i,t_{i+1}]$; in other words, $f$ is the linear interpolation between the nodes $(t_0,x_0),\dots,(t_k,x_k)$.
Clearly it is possible to play with the variables in order to make $f$ a Wajsberg function.

Let's see an example. If we want to find a $(2,\emptyset,\emptyset)$-comb  we need to take a function $f$ that has value $1$ only in $\{0,\frac{1}{2},1\}$, so we can take
$$
f=L(0,1;\frac{1}{4},0;\frac{1}{2},1;\frac{3}{4},0;1,1).
$$
This function has integer coefficients in every interval and $f(1)=1$, so it is a Wajsberg function and the identity $f(x)\app 1$ axiomatizes the subvariety of Wajsberg hoops generated by $\alg {\L}_2$.

Now we will heavily use Wajsberg functions and lemma \ref{identities} to prove a couple of fundamental results. From now on, for any reduced triple $P=(I,J,K)$, $\alg B_\Delta$ will be the algebra constructed from $P$ following the directions in Section \ref{bdelta}.

\begin{theorem}\label{embed1} Let $P=(I,J,K)$ be a reduced triple and let $a \in I$; then $\alg {\L}_a$ is embeddable in $\alg B_\Delta$.
\end{theorem}
\begin{proof} If $1 \in I$, then $P = (\{1\},\emptyset, K)$ and if $K = \emptyset$, then $\alg B_\Delta \cong \alg {\L}_1$. So suppose that $K \ne \emptyset$; then by Theorem \ref{thm:bdelta} $\alg B_\Delta$ is isomorphic with the subalgebra of $\alg L_1 \times \alg C_\o$ generated by $(0,c)$.  Consider the Wajsberg function
$$
f(x) = L(0,0;\dfrac{1}{2},1;1,1);
$$
then it is easy to check that $\widehat{f} = (x \imp x^2) \imp x$. Now $f(0)=0$ and $f(1) =1$ in a neighborhood of $1$, hence by lemma \ref{identities} $f(\ol{g}) = (0,1)$ and thus it generates a subalgebra of $\alg B_\Delta$ isomorphic with $\alg L_1$.

Now suppose that $1 \notin I$;  we fix an $a \in I$ and we let $m$  be the product of all the elements of $I \cup J$. Then we consider the Wajsberg function
$$
f(x) = L(0,1;\dfrac{1}{a}-\dfrac{1}{2m},1;\dfrac{1}{a},0;\dfrac{1}{a}+\dfrac{1}{2m},1;1,1).
$$
Now clearly $f(x) = 1$ in any neighborhood of $\frac{n}{m}$ where $n\le m$ and $\frac{n}{m}\ne\frac{1}{a}$ and moreover $f(\dfrac{1}{a}) =0$.  Since $(I,J,K)$ is reduced, $i$ does not divide any element of $I \cup J$ and so
by Lemma \ref{identities}
$$
f(\ol{g}(k,h,i)) = \left\{
                     \begin{array}{ll}
                       0, & \hbox{if $(k,h,i)=(a,1,2)$;} \\
                       1, & \hbox{otherwise.}
                     \end{array}
                   \right.
$$
Hence $\ol{g} \join f(\ol{g})$ generates a subalgebra of $\alg B_\Delta$, isomorphic to the one generated by $\ol{g}(a,1,2)$. But the latter is a generator of $\alg {\L}_a$ so $\alg{\L}_a$ is embeddable in $\alg B_\Delta$, as wished.
\end{proof}

\begin{theorem}\label{embed2} Let $P=(I,\emptyset,\alg C_\o)$; then $\alg C_\o$ is embeddable in $\alg B_\Delta$.
\end{theorem}
\begin{proof} If $I =\emptyset$, then $\alg B_\Delta \cong \alg C_\o$. Otherwise let $m$ be the product of all elements of $I$ and consider the Wajsberg function
$$
f(x) = L(1,1; \dfrac{m-1}{m},1;\dfrac{m}{m+1},\dfrac{m}{m+1};1,1).
$$
Again it is easy to see that $\widehat{f}(x) = x^m \imp x^{m+1}$ and that $f(\frac{n}{m})=1$ for $\dfrac{n}{m}\ne 1$ so that $f(\ol{g}(k,h,i)) = 1$ whenever $i=2$.
In a neighborhood of $1$ we have that $f(x)=x$, so $f(\ol{g}(0,0,3)) = c$, which generates $\alg C_\o$; thus $f(\ol{g})$ generates a subalgebra of $\alg B_\Delta$ that is isomorphic with $\alg C_\o$ and
thus $\alg C_\o$ is embeddable in $\alg B_\Delta$.
\end{proof}

\section{Structurally complete extensions}

By Theorem \ref{main1} and \ref{main2} classifying all the structurally complete finitary (axiomatic) extensions of positive \L ukasiewicz's Logic amounts to describing all the structural quasivarieties (varieties) of Wajsberg hoops.

\subsection{Structural subvarieties}

The {\bf radical} of a  Wajsberg chain $\alg A$, in symbols $\op{Rad}(\alg A)$, is the intersection of the maximal filters of $\alg A$; it is easy to see that $\op{Rad}(\alg A)$ is cancellative and  $\alg A$ is cancellative if and only if $\op{Rad}(\alg A) = \alg A$.  We say that a bounded Wajsberg hoop $\alg A$ {\bf has rank $n$}, if $\alg A/\op{Rad}(\alg A) \cong \alg{\L}_n$. For any bounded Wajsberg hoop  the {\bf divisibility index} $d_\alg A$ of $\alg A$, is the maximum $k$ such that $\alg{\L}_k$ is embeddable in $\alg A$ if any, otherwise $d_\alg A = \infty$.

 Here is a summary of the main results about the rank and the divisibility index; the proofs are either trivial or can be found in  \cite{AglianoMontagna2003} or \cite{Gispert2002}.

\begin{lemma}\label{whproperties} For any $n,k \ge 1$
\begin{enumerate}
\item $\alg {\L}_n$ is simple and $\alg{\L}_n \in \II\SU(\alg{\L}_k)$ if and only if $\alg {\L}_n \in \II\SU(\alg {\L}_k^\infty)$ if and only if $n \mathrel{|} k$.
\item $\alg {\L}_n$ has rank $n$ and divisibility index $n$.
\item For any $k\ge 0$, $\alg {\L}_{n,k}$ is subdirectly irreducible,  $\alg{\L}_{n,k}$ has rank $n$ and $d_{\alg {\L}_{n,k}} = \gcd(n,k)$; in particular $d_{\alg {\L}^\infty_n} = n$.
\item If $\alg A$ has rank $n$, then $\alg A \in \II\SU\PP_u(\alg {\L}_{n,k})$ if and only if $d_\alg A$ divides $\gcd(n,k)$.
\item If $\alg A$ has rank $n$, then $\alg {\L}_{n,k} \in \II\SU\PP_u(\alg A)$ if and only if $\gcd(n,k)$ divides $d_\alg A$.
\item If $\alg A$ is a nontrivial totally ordered cancellative hoop  then $\II \SU\PP_u(\alg A) = \II\SU\PP_u(\alg C_\o)$.
\item  If $\alg A$ is a bounded Wajsberg chain of finite rank
$k$,  then $d_\alg A$ divides $k$, and $\II\SU \PP_u(\alg A)= \II\SU\PP_u(\alg {\L}_{k,d_\alg A})$.
\item  If $\alg A$ is a bounded Wajsberg chain of finite rank $n$, then $\II\SU\PP_u(\alg A) = \II\SU\PP_u(\alg {\L}^\infty_n)$ if and only if $d_\alg A= n$.
\end{enumerate}
\end{lemma}

We have:

\begin{theorem} \cite{Agliano2023} Let $\alg A_1,\dots,\alg A_n$ be Wajsberg chains; if for all $i \le n$
\begin{enumerate}
\ib $\alg A_i$ is finite, or
\ib $\alg A_i$ is cancellative, or
\ib $\alg {\L}_n \in \II\SU(\alg A_i)$ for all $n$, or
\ib $\alg A$ is infinite, bounded and the rank of $\alg A$ is equal to $d_\alg A$,
\end{enumerate}
then $\QQ(\alg A_1,\dots,\alg A_n) = \VV(\alg A_1,\dots,\alg A_n)$.  Moreover if  $n=1$ then the converse holds as well.
\end{theorem}

This shows that for every reduced presentation $P$, $\vv V(P)= \vv Q(P)$ and henceforth a proper subvariety $\vv V(P)$ of Wajsberg hoop is structural if and only if $\vv V(P) = \QQ(\alg F_{\vv V(P)}(x)$. First we will consider locally finite varieties.

It is clear from the definition that both $\alg C_\o$ and $\alg {\L}_j^\infty$ contain finitely generated subalgebras that fail to be finite. Hence from Theorem \ref{aglianopanti} we get that a variety $\vv V$ of Wajsberg hoops is locally finite if and only if it is $\vv V (P)$ where $P=(I,\emptyset,\emptyset)$ if and only if it is finitely generated.

Now Wajsberg hoops are basic hoops and:

\begin{theorem} \cite{AglianoUgolini2022} Every finite basic hoop is projective in the class of finite basic hoops. So if $\vv V$ is a locally finite variety of basic hoops every finite algebra is projective.
\end{theorem}

Therefore:

\begin{theorem}\label{lfquasivarieties} Every locally finite quasivariety of Wajsberg hoops is a primitive variety.
\end{theorem}
\begin{proof} Let $\vv Q$ be a locally finite quasivariety of Wajsberg hoops; then it is easy to check that $\vv V = \HH(\vv Q)$ is locally finite and hence every finite algebra of $\vv V$ is projective in $\vv V$.
By Theorem \ref{maingorbunov} $\vv V$ is primitive and thus every subquasivariety of $\vv V$ is equational, i.e. it is a variety. This implies that $\vv Q$ is a variety and in fact $\vv Q = \vv V$.
\end{proof}

Via the Blok-Pigozzi connection we get:

\begin{corollary} Every locally tabular extension of $\mathcal {MV}^+$ is tabular, axiomatic and hereditarily structurally complete.
\end{corollary}
\begin{proof}  By Theorem \ref{aglianopanti} every proper subvariety of Wajsberg hoops is of the form $\vv V(I,J,K)$ for some reduced presentation.
It is obvious that $\alg C_\o$ is not locally finite and it is easy to see that the same holds for $\alg {\L}_n^\infty$ for $n\ge 1$. Therefore if $\vv V(I,J,K)$ is locally finite, then $J=K=\emptyset$. But $I$ is a finite set, so $\vv V(I,\emptyset,\emptyset)$ is finitely generated. So every locally finite variety of Wajsberg hoops is finitely generated; so every locally tabular extension of $\mathcal {MV}^+$ is tabular and the rest follows from Theorem \ref{lfquasivarieties}.
\end{proof}

The variety $\vv C =\vv V(\emptyset,\emptyset,\{0\})$ is the variety of {\bf cancellative hoops}; now $\vv C$ can be shown to be primitive by a variety of means. The simplest one is probably to observe first that it is an atom
in the lattice of subvarieties $\Lambda(\mathsf{WH})$ \cite{Amer1984}  hence it is equationally complete. Then one can quote \cite{BergmanMcKenzie1990} where it is stated that any equationally complete congruence modular variety
has no proper subquasivarieties. As $\vv C$ is congruence distributive (having a lattice reduct) it has no proper subquasivarieties and hence it is primitive.

Now we can characterize all the structural varieties of Wajsberg hoops. First we observe that $\WH$ itself is not structural; this is well-known and can be shown in several ways. The most direct one is probably to observe that
the set $\{\alg {\L}_p: \text{$p$ prime}\}$ consists of simple algebras with no proper subalgebras and then invoke Corollary 1 in \cite{AdamsDziobiak1994}.

\begin{theorem} Let $P=(I,J,K)$ be a reduced triple such that $\vv V =\vv V(I,J,K)$ is a proper subvariety of Wajsberg hoops. Then $\vv V$ is structural if and only if either $J =\emptyset$, or $J =\{1\}$.
\end{theorem}
\begin{proof} Suppose  that  $J\ne \emptyset$ and $J \ne \{1\}$.
Then there is an $n \in J$ with $n > 1$.  Let $\vv K = \{\alg {\L}_i: i \in I\} \cup \{\alg {\L}_j: j \in J, j \ne n\} \cup \{\alg {\L}_{n,1}\}$;
Clearly $\QQ(\vv K) \sse \vv Q(P) = \vv V(P)$; if $\alg {\L}_n^\infty \in \QQ(\vv K)$ then, by Theorem \ref{birkhoff},  $\alg {\L}_n^\infty \in \II\SU\PP_u(\vv K)$, since it is subdirectly irreducible.
But all the chains in $\vv K$ are either finite or their divisibility index is not divided by $n$; hence, by Lemma \ref{whproperties} (3) and (4),  $\alg {\L}_n^\infty \notin \II\SU\PP_u(\vv K)$.
Therefore $\QQ(\vv K) \subsetneq \vv Q(P)$; however, as $\alg {\L}_{n,1} \in \VV(\alg {\L}_n^\infty$), $\HH(\QQ(\vv K)) = \vv V(P)$.  So $\vv V(P)$ is not structural.

For the converse, modulo the results on locally finite varieties above, we need only to prove that $\vv V(P)$ is structurally complete whenever $P=(I,\emptyset,\{\o\})$ or $P=(I,\{1\},\emptyset)$.
In either case,  by Theorems  \ref{embed1} and \ref{embed2}, every generator of $\vv V(P)=\vv Q(P)$ is embeddable in $\alg B_\Delta$.  So
$$
\vv V(P) = \vv Q(P) \sse \QQ(\alg B_\Delta) = \QQ(\alg F_{\vv V(P)}(x))
$$
and this proves that $\vv V(P)$ is structural.
\end{proof}

By the description of the proper subvarieties of $\WH$ by reduced triples, we get at once:

\begin{corollary}\label{primitivevar} A variety of Wajsberg hoops is structural if and only if it is primitive.
\end{corollary}

And thus, via the Blok-Pigozzi connection:

\begin{corollary} An axiomatic extension on $\mathcal{MV}^+$ is structurally complete if and only if it is hereditarily structurally complete.
\end{corollary}

\subsection{Structural subquasivarieties} For quasivarieties we need to work a little bit more. First we need a lemma that appears in \cite{Gispert2016} (Lemma 4.5):

\begin{lemma}\label{gispert} Let $n>1$ and let $\alg D_n$ be the subalgebra of $\alg{\L}_{n,1} \times \alg{\L}_{n,n-1}$ generated by $((1,0),(1,1))$. Then $\alg {\L}_{n,1}$ is embeddable in $\alg D_n$.
\end{lemma}

Using Lemma \ref{gispert} we can prove:

\begin{lemma}\label{embed3} Let $P=(I,J,\emptyset)$ be a reduced triple. Then for any $j\in J$, $\alg{\L}_{j,1}$ is embeddable in $\alg B_\Delta$.
\end{lemma}
\begin{proof}

If $J=\emptyset$ we have nothing to prove.

If $1\in J$, then since $P$ is a reduced triple, $P=(I,\{1\},\emptyset)$. Let $m$ be the product of all the elements of $I$ (if $I=\emptyset$ take $m=1$) and consider the Wajsberg function
$$
f(x) = L(0,0;\frac{1}{3m},0;\frac{2}{3m},1;1,1).
$$
This function has value $0$ in a neighborhood of $0$ and has value $1$ for every $\frac{n}{m}\neq 0$, so, by Lemma \ref{identities}, $f(\ol{g}(k,h,i))=0$ if $(k,h,i)=(1,0,0)$, otherwise $f(\ol{g}(k,h,i))=1$. Since $\ol{g}(1,0,0)$ is a generator of $\alg{\L}_1^\infty$, $\ol{g}\join f(\ol{g})$ generates a subalgebra of $\alg{B}_\Delta$ isomorphic to $\alg{\L}_1^\infty$.
Now, by Lemma \ref{whproperties}, $\vv Q(\alg{\L}_1^\infty)=\vv Q(\alg{\L}_{1,1})$; hence $\alg {\L}_{1,1}$ is embeddable into $\alg{\L}_1^\infty$ and thus  into $\alg{B}_\Delta$.

Now suppose $1\notin J$ and fix $j\in J$. Let $m$ be the product of every element of $I\cup J$ and consider the  Wajsberg function
$$
f(x) = L(0,1;\frac{1}{j}-\frac{2}{3m},1;\frac{1}{j}-\frac{1}{3m},0;\frac{1}{j}+\frac{1}{3m},0;\frac{1}{j}+\frac{2}{3m},1;1,1).
$$
This function has value $0$ in a neighborhood of $\frac{1}{j}$ and has value $1$ in a neighborhood of $\frac{n}{m}$ when $\frac{n}{m}\neq\frac{1}{j}$. By Lemma \ref{identities}, $f(\ol{g}(k,h,i))=0$ if $(k,h,i)=(j,1,0)$ or $(k,h,i)=(j,j-1,0)$. Now let us show that $f(\ol{g}(k,h,i))=1$ in all the other cases.

Take $(k,h,i)\neq(j,1,0),(j,j-1,0)$. If $i\in\{0,1\}$, then $\ol{g}(k,h,i)=(r,s)$ for some $(r,s)\in\alg{\L}_{k,h}$; notice that, by construction of $f$, $f(\ol{g}(k,h,i))\neq 1$ only if $\frac{r}{k}=\frac{1}{j}$, but this happens only if $jr=k$, but this can not happen because the triple is reduced, so $k$ can not be a multiple of $j$. If $i=2$, then $\ol{g}(k,h,i)=h$ for some $0\leq h<k$ and $k,h$ relatively prime; this time $f(\ol{g}(k,h,i))\neq 1$ only if $\frac{h}{k}=\frac{1}{j}$, that is only if $j=\frac{k}{h}$, but this means that $h$ divides $k$, which is not possible because $k,h$ are relatively prime.

Thus, if we consider $\ol{g}\join f(\ol{g})$, this generates a subalgebra of $\alg{B}_\Delta$ isomorphic to $\alg{D}_j$ as defined in Lemma \ref{gispert}; moreover, by the same lemma, we get that $\alg{L}_{j,1}$ is embeddable into $\alg{D}_j$ and hence  into $\alg{B}_\Delta$.
\end{proof}

Let $P=(I,J,K)$ be a triple (not necessarily reduced) and let $\vv Q[I,J,K]$ be defined in the following way
\begin{align*}
&\vv Q[I,\emptyset,K] = \vv Q(I,\emptyset,K) \\
&\vv Q[I,J,\emptyset] = \QQ(\{\alg{\L}_i: i \in I\} \cup \{\alg{\L}_{j,1}: j \in J\})\quad\text{if $J \ne \emptyset$}.
\end{align*}

\begin{theorem}\label{mainquasi} Let $\vv Q$ be a quasivariety of Wajsberg hoops; then $\vv Q$ is structural if and only if it is either $\vv Q= \QQ(\alg F_{\WH}(x))$ or else $\vv Q = \vv Q[P]$ for some reduced triple $P$.
\end{theorem}
\begin{proof}  As the structurally complete subquasivarieties are exactly $\QQ(\alg F_{\vv V}(x))$ for $\vv V \sse \WH$, it is enough to show that for every reduced triple $P = (I,J,K)$,
 $\vv Q[P] = \QQ(\alg F_{\vv V(P)}(x))$. Moreover  if $\vv V$ is a subvariety of $\WH$, then $\QQ(\alg F_\vv V(x))$ is the structural core of $\vv V$, i.e. the smallest subquasivariety $\vv Q$ of $\vv V$ such that $\VV(\vv Q)=\vv V$. Now for any reduced presentation $P$, we clearly have that $\VV(\vv Q[P]) = \VV(P)$; so to get the conclusion it is enough to prove that $\vv Q[P] \sse \QQ(\alg F_{\vv V(P)}(x))$.

But Theorem \ref{embed1}, \ref{embed2} and Lemma \ref{embed3} show that any generator of $\vv Q[P]$ is embeddable in $\alg B_\Delta$; so
$$
\vv Q[P] \sse \vv Q(\alg B_\Delta) = \QQ(\alg F_{\vv V(P)}(x))
$$
and the conclusion holds.
\end{proof}

Observe that $\QQ(\alg {\L}_{1,1}) = \QQ(\alg {\L}_1^\infty)$; so if $P=(I,J,K)$ is such that either $J = \emptyset$ or $J=\{1\}$, then $Q[P] = Q(P)$ and they are all in fact primitive varieties, by Corollary \ref{primitivevar}.
We can make another observation of some relevance based on the results in \cite{CzelakowskiDziobiak1990}. It is well known (and easy to prove) that the Wajsberg chains coincide with the finitely subdirectly irreducible Wajsberg hoops and the variety of Wajsberg hoops is congruence distributive. Then:
\begin{enumerate}
\ib every structural subquasivariety  $\vv Q$ is generated by finitely subdirectly irreducible Wajsberg hoops and hence it is relatively congruence distributive;
\ib in any structural subquasivariety $\vv Q$ the finitely $\vv Q$-irreducible algebras are finitely subdirectly irreducible in the absolute sense, i.e. they are all Wajsberg chains;
\ib hence any algebra $\alg A \in \vv Q$ is subdirectly embeddable in a product of Wajsberg chains that belong to $\vv Q$.
\end{enumerate}
This information might be useful for characterizing the structural subquasivarieties  $\vv Q[P]$ that are primitive (besides the ones that are already known).  We will consider the quasivarieties
$\vv Q[I,J,K]$ where $K =\emptyset$ and $J \ne \emptyset,\{1\}$ and  in this case we will write simply $\vv Q
[I,J]$.

\begin{lemma}\label{orderlemma}
Let $(I,J,\emptyset)$ and $(I',J',\emptyset)$ be two triples (not necessarily reduced), then $\vv Q[I,J] \subseteq \vv Q[I',J']$ if and only if for every $i\in I$ s.t. $i\neq 1$ and for every $j \in J$, there are $i'\in I'$ and $j'\in J'$ with $i|i'$ and $j|j'$.
\end{lemma}

\begin{proof}
If $i=1$, then $\alg{\L}_1\in\II\SU(\alg{\L}_n)$ and $\alg{\L}_1\in\II\SU(\alg{\L}_{n,1})$ for every $n$. Take now $i\neq 1$ and suppose that $\alg {\L}_i \in \II\SU\PP_u(\{\alg {\L}_{i'}: i\ \in I'\} \cup \{\alg {\L}_{j',1}: j'\in J'\})$; then $\alg {\L}_i \notin \II\SU\PP_u(\{\alg {\L}_{j',1}: j'\in J'\})$
as $\alg {\L}_i$ has divisibility index $i \ne 1$ and each $\alg {\L}_{j',1}$ has divisibility index equal to $1$.  On the other hand no $\alg {\L}_{j,1}$ can belong to
$\II\SU\PP_u(\{\alg {\L}_{i'}: i\ \in I'\})$ for obvious reasons.  From here it is straightforward to check that the conclusion holds.
\end{proof}

Now we can give an example of a class of quasivarieties that are not primitive.

\begin{proposition}\label{notprimitivequasivariety}
Let $\vv Q=\vv Q[I,J]$ with $(I,J,\emptyset)$ reduced triple; if $I{\downarrow}\cap J{\downarrow}\supsetneq\{1\}$, then $\vv Q$ is not primitive.
\end{proposition}

\begin{proof}
If $I{\downarrow}\cap J{\downarrow}\supsetneq\{1\}$, then there exists $n\neq 1$ in $I{\downarrow}\cap J{\downarrow}$ and, by the previous lemma, $\vv Q[n,n]\subseteq\vv Q[I,J]$. Now, notice that $\vv V(\vv Q[n,n])=\vv V(\emptyset,n,\emptyset)=\vv V(\vv Q[\emptyset,n])$, so, by theorem \ref{mainquasi}, the structural core of $\vv Q[n,n]$ is $\vv Q[\emptyset,n]$. But, by the previous lemma, $\vv Q[\emptyset,n]\subsetneq Q[n,n]$ so, in particular, $\vv Q[n,n]$ is different from its structural core, hence it is not structural. Therefore $\vv Q$ contains a quasivariety that is not structural and this means that it is not primitive.
\end{proof}

We can also give an example of a class of primitive quasivarieties.

\begin{proposition}\label{primitivequasivariety}
Let $\vv Q=\vv Q[\emptyset,p]$, where $p$ is a prime number; then $\vv Q$ is primitive.
\end{proposition}

\begin{proof}
We know that $\vv V(\vv Q)=\vv V(\emptyset,p,\emptyset)$ and, since $\vv Q$ is structural, this means that every quasivariety strictly contained in $\vv Q$ generates a variety that is strictly contained in $\vv V(\vv Q)$. So now let's consider all the subvarieties of $\vv V(\emptyset,p,\emptyset)$.

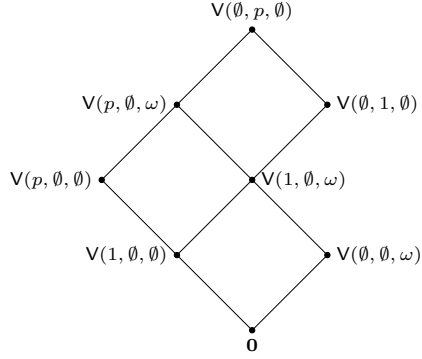
\vspace{0.4cm}
\begin{center}
\begin{tikzpicture}[scale=0.5]
\tikzstyle{every node}=[font=\scriptsize]

\filldraw [black] (5,1) circle (2pt) node[anchor=north]{$\mathbf{0}$};
\filldraw [black] (3,3) circle (2pt) node[anchor=east]{$\vv V(1,\emptyset,\emptyset)$};
\filldraw [black] (7,3) circle (2pt) node[anchor=west]{$\vv V(\emptyset,\emptyset,\omega)$};
\filldraw [black] (1,5) circle (2pt) node[anchor=east]{$\vv V(p,\emptyset,\emptyset)$};
\filldraw [black] (5,5) circle (2pt) node[anchor=west]{$\vv V(1,\emptyset,\omega)$};
\filldraw [black] (3,7) circle (2pt) node[anchor=east]{$\vv V(p,\emptyset,\omega)$};
\filldraw [black] (7,7) circle (2pt) node[anchor=west]{$\vv V(\emptyset,1,\emptyset)$};
\filldraw [black] (5,9) circle (2pt) node[anchor=south]{$\vv V(\emptyset,p,\emptyset)$};
\draw[line width=0.3pt,smooth] (5,1)--(3,3);
\draw[line width=0.3pt,smooth] (5,1)--(7,3);
\draw[line width=0.3pt,smooth] (3,3)--(1,5);
\draw[line width=0.3pt,smooth] (3,3)--(5,5);
\draw[line width=0.3pt,smooth] (7,3)--(5,5);
\draw[line width=0.3pt,smooth] (1,5)--(3,7);
\draw[line width=0.3pt,smooth] (5,5)--(3,7);
\draw[line width=0.3pt,smooth] (5,5)--(7,7);
\draw[line width=0.3pt,smooth] (3,7)--(5,9);
\draw[line width=0.3pt,smooth] (7,7)--(5,9);

\end{tikzpicture}
\end{center}

If we consider a quasivariety $\vv Q'$ that generates a variety strictly contained in $\vv V(\emptyset,p,\emptyset)$, then it has to be contained in the coatoms of the lattice, that are $\vv V(p,\emptyset,\omega)$ and $\vv V(\emptyset,1,\emptyset)$. By theorem 4.6, we know that these two varieties are structural, hence primitive, so every quasivariety contained in them is structural.
\end{proof}

Unfortunately, we fall short of characterizing all the primitive quasivarieties, due to our lack of understanding of the lattice of all the subquasivarieties. The proof of the Proposition \ref{primitivequasivariety} is based on the fact that all the quasivarieties strictly contained in $\vv Q[\emptyset,p]$ are actually varieties; so we know that the lattice of the subquasivarieties of $\vv Q$ is contained in the lattice of the subvarieties of $\vv V(\vv Q)$ which clearly is not always the case.

Consider for example the quasivariety $\vv Q=\vv Q[\emptyset,\{p,q\}]$. Using Lemma \ref{orderlemma}, we may sketch what the lattice of the subquasivarieties of $\vv Q$ looks like:

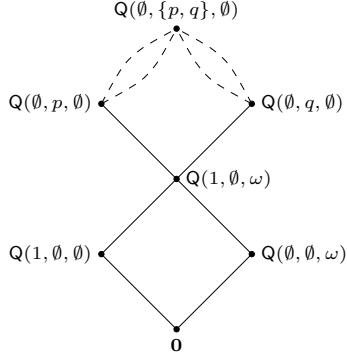
\vspace{0.5cm}
\begin{center}
\begin{tikzpicture}[scale=0.5]
\tikzstyle{every node}=[font=\scriptsize]

\filldraw [black] (5,1) circle (2pt) node[anchor=north]{$\mathbf{0}$};
\filldraw [black] (3,3) circle (2pt) node[anchor=east]{$\vv Q(1,\emptyset,\emptyset)$};
\filldraw [black] (7,3) circle (2pt) node[anchor=west]{$\vv Q(\emptyset,\emptyset,\omega)$};
\filldraw [black] (5,5) circle (2pt) node[anchor=west]{$\vv Q(1,\emptyset,\omega)$};
\filldraw [black] (3,7) circle (2pt) node[anchor=east]{$\vv Q(\emptyset,p,\emptyset)$};
\filldraw [black] (7,7) circle (2pt) node[anchor=west]{$\vv Q(\emptyset,q,\emptyset)$};
\filldraw [black] (5,9) circle (2pt) node[anchor=south]{$\vv Q(\emptyset,\{p,q\},\emptyset)$};
\draw[line width=0.3pt,smooth] (5,1)--(3,3);
\draw[line width=0.3pt,smooth] (5,1)--(7,3);
\draw[line width=0.3pt,smooth] (3,3)--(5,5);
\draw[line width=0.3pt,smooth] (7,3)--(5,5);
\draw[line width=0.3pt,smooth] (5,5)--(3,7);
\draw[line width=0.3pt,smooth] (5,5)--(7,7);
\draw[line width=0.3pt,dashed] (3,7)..controls(3.7,8.3)..(5,9);
\draw[line width=0.3pt,dashed] (3,7)..controls(4.3,7.7)..(5,9);
\draw[line width=0.3pt,dashed] (7,7)..controls(6.3,8.3)..(5,9);
\draw[line width=0.3pt,dashed] (7,7)..controls(5.7,7.7)..(5,9);

\end{tikzpicture}
\end{center}

Now, $\vv Q(\emptyset,p,\emptyset)$ and $\vv Q(\emptyset,q,\emptyset)$ are primitive by Lemma \ref{primitivequasivariety}, but we don't know if there is any quasivariety in the intervals $[\vv Q(\emptyset,p,\emptyset),\vv Q(\emptyset,\{p,q\},\emptyset)]$ and $[\vv Q(\emptyset,q,\emptyset),\vv Q(\emptyset,\{p,q\},\emptyset)]$. Note that, if there is such a quasivariety, then it cannot be generated by chains, and that would immediately imply that $\vv Q$ is not primitive, because by Theorem \ref{mainquasi} this quasivariety would not be structural.

Another problem is that, given a quasivariety $\vv Q$, the lattice of all the varieties $\vv V(\vv Q')$, with $\vv Q'\subseteq\vv Q$, is almost always strictly contained in the lattice of the subvarieties of $\vv V(\vv Q)$.
For example consider $\vv Q=\vv Q[\emptyset,pq]$; then $\vv V(\vv Q)=\vv V(\emptyset,pq,\emptyset)$ and we know that $\vv V(pq,p,\emptyset)$ is a subvariety of $\vv V(\emptyset,pq,\emptyset)$. Now, if a quasivariety generates $\vv V(pq,p,\emptyset)$ then it would not be primitive, because it would contain $\vv Q[pq,p]$ that is not primitive by Lemma \ref{notprimitivequasivariety}, but by lemma \ref{orderlemma} we know that no subquasivariety of $\vv Q$ can contain $\alg {\L}_{pq}$, so there is no subquasivariety of $\vv Q$ that can generate $\vv V(pq,p,\emptyset)$.

\section{Conclusions and future work}

What can we say about the fragments of $\mathcal{MV}^+$?  We will consider only the fragments containing $\imp$, as they are the algebraizable ones.
The $\{\imp\}$-fragment has been studied first in \cite{Komori1978}; its equivalent algebraic semantics is the variety $\mathsf{LBCK}$ of {\bf \L ukasiewicz $\mathsf{BCK}$ algebras}. We have that:
\begin{enumerate}
\ib every locally finite quasivariety of $\mathsf{LBCK}$-algebras is a primitive variety \cite{AglianoUgolini2022};
\ib the only non-locally finite subvariety is the entire variety $\mathsf{LBCK}$ \cite{Komori1978};
\ib $\mathsf{LBCK}$  it is generated as a quasivariety by its finite chains \cite{AFM};
\ib every infinite chain contains all the finite chains as subalgebras \cite{Komori1978};
\ib so if $\vv Q$ is a quasivariety which contains only finitely many chains, then $\VV(\vv Q)$ is locally finite, hence primitive;
\ib otherwise $\vv Q$ contains infinitely many chains and so $\VV(\vv Q) =  \vv Q = \mathsf{LBCK}$.
\end{enumerate}
Hence every subquasivariety of $\mathsf{LBCK}$ is a variety and $\mathsf{LBCK}$ is primitive.

For the other algebraizable fragments, observe that if $\imp$  is present then $\join$ is definable and if $\imp$ and $\cdot$ are present, then $\meet$ is definable. So the only remaining interesting fragment is the  $\{\imp, \meet, 1\}$-fragment that has been considered in \cite{Agliano2022}. Its equivalent algebraic semantics is the variety $\mathsf{LBCK}^\meet$ of $\mathsf{LBCK}$-semilattices; from the results in \cite{Agliano2022}  (and some straightforward calculations) one can prove that $\mathsf{LBCK}^\meet$ is primitive.

As far as the future work is concerned, there is a very natural path to follow. In this paper we have characterized all the structurally complete finitary extensions of $\mathcal{MV}^+$ and in \cite{Gispert2016} the same has been basically done for finitary extensions of $\mathcal{MV}$.  Hajek's Basic Logic $\mathcal{BL}$ \cite{Hajek1998} has been investigated from the algebraic point of view in many papers through its equivalent algebraic semantics, that is the variety of $\mathsf{BL}$-algebras (see for instance \cite{Agliano2017c} and the bibliography therein). In particular, in \cite{AglianoMontagna2003} (Theorem 3.7) it has been shown that there is a very deep algebraic connection between $\mathsf{BL}$-algebras (and their positive subreducts), $\mathsf{MV}$-algebras and Wajsberg hoops.
With the knowledge we have accumulated so far (and using also more general techniques introduced in \cite{AglianoUgolini2022}) we believe we can tackle the problem of describing the finitary structurally complete extensions of $\mathcal{BL}$ with some degree of success.

\providecommand{\bysame}{\leavevmode\hbox to3em{\hrulefill}\thinspace}
\providecommand{\MR}{\relax\ifhmode\unskip\space\fi MR }
\providecommand{\MRhref}[2]{%
  \href{http://www.ams.org/mathscinet-getitem?mr=#1}{#2}
}
\providecommand{\href}[2]{#2}

\end{document}